\documentclass{amsart}
\usepackage{amssymb, amsmath, amsthm}
\usepackage[alphabetic]{amsrefs}

\usepackage[colorlinks, citecolor = blue]{hyperref}
\usepackage{graphicx}

\usepackage{enumerate}

\usepackage{array}   
\newcolumntype{L}{>{$}l<{$}} 
\newcolumntype{C}{>{$}c<{$}} 

\numberwithin{equation}{section}
\newtheorem{prop}{Proposition}
\numberwithin{prop}{section}
\newtheorem{thm}[prop]{Theorem}
\newtheorem{lemma}[prop]{Lemma}
\newtheorem{cor}[prop]{Corollary}

\theoremstyle{remark}

\newcommand{\Q}{\mathbb Q}
\newcommand{\Z}{\mathbb Z}

\newcommand{\C}{\mathbb C}
\newcommand{\A}{\mathbb A}
\newcommand{\E}{\mathbb E}
\newcommand{\F}{\mathbb F}

\newcommand{\Fr}{\mathrm{Fr}}
\newcommand{\Aut}{\mathrm{Aut}}

\newcommand{\Cent}{\mathrm{Cent}}
\renewcommand{\~}{\tilde}

\newcommand{\bmx}{\begin{pmatrix}}
\newcommand{\emx}{\end{pmatrix}}

\DeclareMathOperator{\GL}{GL}
\DeclareMathOperator{\PGL}{PGL}
\DeclareMathOperator{\SL}{SL}
\DeclareMathOperator{\PSL}{PSL}
\DeclareMathOperator{\Gal}{Gal}
\DeclareMathOperator{\tr}{tr}

\begin{document}

\title[Distinguishing finite group characters]{Distinguishing finite group characters
and refined local-global phenomena}
\date{\today}

\author{Kimball Martin}
\address{Department of Mathematics, University of Oklahoma, Norman, OK 73019 USA}
\author{Nahid Walji}
\address{Department of Mathematics, Occidental College, Los Angeles, CA 90041 USA}

\maketitle

\begin{abstract}
Serre obtained a sharp bound on how often two irreducible degree $n$ complex
characters of a finite group can agree, which tells us how many local
factors determine an Artin $L$-function.  We consider the more delicate question of finding
a sharp bound when these objects are primitive, and answer these questions for $n=2,3$.
This provides some insight on refined strong multiplicity one phenomena for automorphic 
representations of $\GL(n)$.  For general $n$, we also answer the character question for the 
families $\PSL(2,q)$ and $\SL(2,q)$.  
\end{abstract}

\section{Introduction}

In this paper, we consider two questions about seemingly different topics:

\begin{enumerate}
\item How often can two characters of a finite group agree?

\item How many local Euler factors determine an $L$-function?
\end{enumerate}

The first question is just about characters of finite groups, and the second is a
refined local-global principle in number theory.  However, it has been observed, notably by 
Serre, that being able to say something about (1) allows one to say something 
about (2), which is our primary motivation, though both are natural questions.
Our main results about the first question are for comparing
primitive characters of degree $\le 3$ and characters of $\PSL(2,q)$ or $\SL(2,q)$.
This will yield sharp bounds on how many Euler factors one needs to
distinguish primitive 2- or 3-dimensional $L$-functions of Galois representations.
We address them in turn.

\subsection{Distinguishing group characters}
Let $G$ be a finite group, and $\rho, \rho'$ be two complex representations of 
$G$ with characters $\chi, \chi'$.  We will study the quantities
\[ \delta(\rho, \rho') = \delta(\chi, \chi') = 
 \frac  {| \{ g \in G : \chi(g) \ne \chi'(g) \} |}{|G|}. \]

Specifically, let $\delta_n(G)$ be the minimum of $\delta(\rho, \rho')$ as $\rho, \rho'$ range
over pairs of inequivalent irreducible $n$-dimensional representations of $G$,
with the convention that $\delta_n(G) = 1$ if there are no such pairs $\rho, \rho'$.
Note that $\delta_n(G)$ tells us what fraction of elements of $G$ we must check 
to distinguish irreducible degree $n$ characters.  Put 
$d_n = \inf_G \{ \delta_n(G) \}$.

An elementary consequence of orthogonality relations is
\begin{prop} \label{prop1} We have $d_n \ge  \frac 1{2n^2}$.
\end{prop}

Buzzard, Edixhoven and Taylor constructed examples to show this bound
is sharp when $n$ is a power of 2, which Serre generalized this to arbitrary $n$
(see \cite{ramakrishnan:motives}).

\begin{thm}[Serre] \label{thm:serre}
For any $n$, there exists $G$ such that $\delta_n(G) =  \frac 1{2n^2}$,
so $d_n = \frac 1{2n^2}$.
\end{thm}

In particular, the infimum in $d_n$ is a minimum.
We will recall the proof of Proposition \ref{prop1} and Serre's construction in Section
\ref{sec:serre}.  For now, the main points to note are that Serre's examples must be solvable
and the representations are induced.

In this paper, we consider two kinds of refinements of determining $d_n$.
The first refinement is about restricting to primitive representations and the second
is about restricting to certain families of groups.

Define $\delta^\natural_n(G)$ to be the infimum of $\delta(\rho, \rho')$
where $\rho, \rho'$ range over pairs of inequivalent
irreducible primitive $n$-dimensional complex representations of $G$.
Let $d^\natural_n = \inf_G \{ \delta_n^\natural(G) \}$.  
From Serre's theorem, we get a trivial bound $d^\natural_n \ge d_n = \frac 1{2n^2}$.

Our first result is to determine $d^\natural_n$ for $n \le 3$.

\begin{thm} \label{thm1} We have $d_1^\natural = \frac 12$, $d_2^\natural = \frac 14$
and $d_3^\natural = \frac 27$.  Furthermore, $\delta^\natural_2(G) = \frac 14$ if and only if 
$G$ is an extension of $H \times_{C_2} C_{2m}$ where $m \in \mathbb N$
and $H=[48,28]$ or $H=[48,29]$.  Also,
$\delta^\natural_3(G) = \frac 27$ if and only if $G$ is an extension of $\PSL(2,7)$.
\end{thm}

Here $G$ being an extension of $H$ by some $N \lhd G$
means $G/N \simeq H$.
The groups $[48,28]$ and $[48,29]$ are the two groups of order 48 
which are extensions of $S_4$ by the cyclic group $C_2$ and contain $\SL(2,3)$.

The $n=1$ case is already contained in Proposition \ref{prop1} as $d_1 = d_1^\natural$.  
For $n=2, 3$, these bounds are much better than the trivial bounds 
$d_2^\natural \ge \frac 18$ and $d_3^\natural \ge \frac 1{18}$ from
Proposition \ref{prop1}.  For $n=2$, related results were
previously obtained by the second author in \cite{walji} and will be discussed below.

Note that while $d_n$ is a strictly decreasing sequence for $n \ge 1$, our result says
this is not the case for $d_n^\natural$.

\medskip
In a slightly different direction, one can look for stronger lower bounds
than $\frac 1{2n^2}$ for certain families of groups.  We do not begin a serious investigation
of this here, but just treat two basic families of finite groups of Lie type which are related
to the calculations for $\delta^\natural_2(G)$ and $\delta^\natural_3(G)$.

\begin{thm} \label{thm2}
We compute $\delta_n(G)$ and
$\delta^{\natural}_n(G)$ where $G = \PSL(2,q)$ and $G=\SL(2,q)$;
 for $n$ not listed explicitly below, $\delta_n(G) = \delta^{\natural}_n(G)=1$.

For $G = \SL(2,q)$ with $q$ arbitrary or for $G = \PSL(2,q)$ with $q$ even,
\begin{align*}
\delta_n(G) = \delta^{\natural}_n (G)  \
     \left\{ \begin{array}{cll}
= \frac{2}{q}& \text{if }n = \frac{q \pm 1}{2} \text{ and $q$ is odd},& \\
\ge \frac{1}{6}& \text{if }n = q - 1,&\\
     \end{array} \right.
\end{align*}
and $\delta_{q+1} (G) \ge \frac{1}{6}$ whereas $\delta^{\natural}_{q+1} (G) =  1$.

For $G = \PSL(2,q)$ and $q$ odd,
\begin{align*}
\delta_n(G) = \delta^{\natural}_n (G)  \  \left\{ \begin{array}{cll}
= \frac{2}{q}&\text{if } n = \frac{q - 1}{2} \text{ and $q \equiv 3 \bmod 4$},& \\
= \frac{2}{q}&\text{if } n = \frac{q + 1}{2} \text{ and $q \equiv 1 \bmod 4$},& \\
\ge \frac{1}{6}&\text{if } n = q - 1,&\\
     \end{array} \right.
\end{align*}
and $\delta_{q+1} (G) \ge  \frac{1}{6}$ whereas $\delta^{\natural}_{q+1} (G) =  1$.
\end{thm}

We remark that we completely determine $\delta_{q \pm 1}(G)$ for $G=\SL(2,q)$
and $\PSL(2,q)$ in Section \ref{sec:sl2q}, but the exact formulas are a bit complicated and 
depend on divisibility conditions of $q \mp 1$.  In particular,  
$\delta_{q \pm 1}(\SL(2,q)) = \frac 16$ if and only if $12 | (q \mp 1)$,
and $\delta_{q \pm 1}(\PSL(2,q)) = \frac 16$ if and only if $24 | (q \mp 1)$.

The values for $\SL(2,q)$ immediately give the following bounds.

\begin{cor} $d_{(q\pm 1)/2}^\natural \leq \frac{2}{q}$ for $q$ any odd prime power greater than 3.
\end{cor}

Note the upper bound in the corollary for $q=7$ is the exact value of $d^\natural_3$.

Even though Theorem \ref{thm1} implies $d_n^\natural$ is not a decreasing sequence for
$n \ge 1$, this corollary at least suggests that $d_n^\natural \to 0$ as $n \to \infty$.

The proof of Theorem \ref{thm1} relies on consideration of various cases according to
the possible finite primitive subgroups of $\GL_2(\C)$ and $\GL_3(\C)$ which are ``minimal lifts'', and about half of these are of the form $\PSL(2,q)$ or $\SL(2,q)$ 
 for $q \in \{ 3, 5, 7, 9 \}$.  Thus Theorem \ref{thm2} is a generalization of one of the ingredients for Theorem \ref{thm1}.  However, most of
the work involved in the proof of Theorem \ref{thm1} is the determination of and
reduction to these minimal lifts, as described in Section \ref{sec:general}.

\subsection{Distinguishing $L$-functions}
Let $F$ be a number field, and consider an $L$-function $L(s)$, which is a meromorphic
function of a complex variable $s$ satisfying certain properties, principally having an
Euler product $L(s) = \prod L_v(s)$ where $v$ runs over all primes of $F$ for $s$ in some
right half-plane.  For almost all (all but finitely many)
$v$, we should have $L_v(s) = (p_v(q_v^{-s}))^{-1}$ where $q_v$ is the size of the residue
field of $F_v$ and $p_v$ a polynomial of a fixed degree $n$, which is the degree of the
$L$-function.

Prototypical  $L$-functions of degree $n$ are $L$-functions $L(s, \rho) = \prod L(s, \rho_v)$ 
of $n$-dimensional
Galois representations $\rho : \Gal(\bar F/F) \to \GL_n(\C)$ (or into $\GL_n(\overline{\Q_p})$) 
and $L$-functions $L(s,\pi) = \prod L(s, \pi_v)$ of automorphic representations $\pi$ of 
$\GL_n(\A_F)$.  In fact it is conjectured that all (nice) $L$-functions are automorphic.
These $L$-functions are local-global objects, and one can ask how many local factors
$L_v(s)$ determine $L(s)$.  

First consider the automorphic case: suppose 
$\pi, \pi'$ are irreducible cuspidal automorphic representations of $\GL_n(\A_F)$,
$S$ is a set of places of $F$ and we know that $L(s, \pi_v) = L(s, \pi'_v)$ for all
$v \not \in S$.
Strong multiplicity one says that if $S$ is finite, then $L(s, \pi) = L(s, \pi')$ (in fact, $\pi \simeq \pi'$).  Ramakrishnan \cite{ramakrishnan:motives} conjectured that if $S$ has density 
$< \frac 1{2n^2}$, then $L(s, \pi) = L(s, \pi')$, and this density bound would be sharp.  This is 
true when $n=1$, and Ramakrishnan also showed it when $n=2$ \cite{ramakrishnan:SMO}.

Recently, in \cite{walji} the second author showed that when $n=2$
one can in fact obtain stronger bounds under various assumptions, e.g.,   
the density bound $\frac 18$ from \cite{ramakrishnan:SMO} may be replaced by $\frac 14$ if 
one restricts to non-dihedral representations (i.e., not induced from quadratic
extensions) or by $\frac 29$ if the representations are not  twist-equivalent.

Our motivation for this project was to try to understand an analogue of \cite{walji} for
larger $n$.  However the analytic tools known for $\GL(2)$ that are used in \cite{walji}
are not known for larger $n$.  Moreover, the classification of $\GL(2)$ cuspidal 
representations into dihedral, tetrahedral, octahedral and icosahedral types has no known
nice generalization to $\GL(n)$.  So, as a proxy, we consider the case of Galois (specifically 
Artin) representations.  The strong Artin conjecture says that all Artin representations all 
automorphic, and Langlands' principle of functoriality says that whatever is true for Galois 
representations should  be true (roughly) for automorphic representations as well.

Consider $\rho, \rho'$ be irreducible $n$-dimensional Artin representations for $F$, i.e., 
irreducible $n$-dimensional continuous complex representations of the absolute Galois group 
$\Gal(\bar F/F)$ of $F$.   For almost all places $v$ of $F$, we can associate a well-defined Frobenius conjugacy class $\Fr_v$ of $\Gal(\bar F/F)$, and $L(s, \rho_v)$
determines the eigenvalues of $\rho(\Fr_v)$, and thus $\tr \rho(\Fr_v)$.
Let $S$ be a set of places of $F$, and suppose
$L(s, \rho_v) = L(s, \rho'_v)$, or even just $\tr \rho(\Fr_v) = \tr \rho'(\Fr_v)$, 
for all $v \not \in S$.

  Continuity means that $\rho$ and $\rho'$ factor through
a common finite quotient $G = \Gal(K/F)$ of $\Gal(\bar F/F)$, for some finite normal 
extension $K/F$.  View $\rho, \rho'$ as irreducible $n$-dimensional representations of
the finite group $G$.    The Chebotarev density theorem tells us that if
$C$ is a conjugacy class in $G$, then the image of $\Fr_v$ in $\Gal(K/F)$
lies in $C$ for a set of primes $v$ of density $\frac{|C|}{|G|}$.  This implies that if 
the density of $S$ is $< \delta_n(G)$ (or $< \delta_n^\natural(G)$ if $\rho, \rho'$ are primitive), then $\rho \simeq \rho'$, i.e., $L(s, \rho) = L(s, \rho')$.  
Moreover, this bound on the density of $S$ is sharp.

Consequently, Proposition \ref{prop1} tells us that if the density of $S$ is $< \frac 1{2n^2}$,
then $L(s, \rho) = L(s, \rho')$, and Serre's result implies this bound is sharp.  
(See Rajan \cite{rajan} for an analogous result on $\ell$-adic Galois representations.) 
In fact, this application to Galois representations was Serre's motivation, and it also 
motivated the bound in Ramakrishnan's conjecture. 
For us, the Chebotarev density theorem together with Theorem \ref{thm1} yields

\begin{cor} Let $\rho$, $\rho'$ be  
irreducible primitive $n$-dimensional Artin representations for $F$.  Suppose
$\tr \rho(\Fr_v) = \tr \rho'(\Fr_v)$ for a set of primes $v$ of $F$ of density $c$.

\begin{enumerate}
\item If $n=2$ and $c > \frac 34$, then $\rho \simeq \rho'$.

\item If $n=3$ and $c > \frac 57$, then $\rho \simeq \rho'$.
\end{enumerate}
\end{cor}

When $n=2$, if $\rho$ and $\rho'$ are automorphic, i.e., satisfy the strong Artin
conjecture, then the above result already follows by \cite{walji}. When
$n=2$, the strong Artin conjecture for $\rho$ is known in many cases---for instance,
 if $\rho$ has solvable image by Langlands \cite{langlands}
and Tunnell \cite{tunnell}, or if $F=\Q$ and $\rho$ is ``odd'' via Serre's conjecture
by Khare-Wintenberger \cite{khare-wintenberger}.
We remark that the methods of \cite{walji} are quite different than ours here.

The above corollary suggests the following statement may be true: if $\pi, \pi'$
are cuspidal automorphic representations of $\GL_3(\A_F)$ which are not induced from
characters and $L(s, \pi_v) = L(s, \pi'_v)$ for a set of primes $v$ of density $> \frac 57$,
then $\pi \simeq \pi'$.  Since not all cuspidal $\pi, \pi'$ come from Artin representations,
the $\frac 57$ bound is not even conjecturally sufficient for general $\pi, \pi'$.
However, it seems reasonable to think that coincidences of a large fraction of Euler factors only 
happen for essentially algebraic reasons, so the density bounds are likely to be the same
in both the Artin and automorphic cases.

\subsection*{Acknowledgements}
We thank a referee for pointing out an error in an earlier version.
The first author was partially supported by a Simons Collaboration Grant.
The second author was supported by Forschungskredit grant K-71116-01-01 of the
University of Z\"urich and partially supported by grant SNF PP00P2-138906 of the
Swiss National Foundation.
This work began when the second author visited the first at the University of Oklahoma. The second author would like to thank the first author as well as the mathematics department of the University of Oklahoma for their hospitality. 

\section{Notation and Background}

Throughout, $G$, $H$ and $A$ will denote finite groups, and $A$ will be abelian.
Denote by $Z(G)$ the center of $G$.

If $G$ and $N$ are groups, by a (group) extension of $G$ by $N$ we mean a group $H$
with a normal subgroup $N$ such that $H/N \simeq G$.  The extension is called central
or cyclic if $N$ is a central or cyclic subgroup of $H$.

If $G$, $H$, and $Z$ are groups such that $Z \subset Z(G) \cap
Z(H)$, then the central product $G \times_Z H$ of $G$ and $H$ with respect to $Z$ is defined to be direct product $G \times H$ modulo the central subgroup $\{ (z, z) : z \in Z \}$.

If $\chi_1, \chi_2$ are characters of $G$, their inner product is $(\chi_1, \chi_2) = |G|^{-1} \sum_G \chi_1(g) 
\overline{\chi_2(g)}$. 

We denote a cyclic group of order $m$ by $C_m$.

\subsection{Finite subgroups of $\GL_n(\C)$}

Next we recall some definitions and facts about finite subgroups of $\GL_n(\C)$.

Let $G$ be a finite subgroup of $\GL_n(\C)$, so one has the standard representation
of $G$ on $V = \C^n$.  
We say $G$ is reducible if there exists a nonzero proper subspace 
$W \subset V$ which is fixed by $G$.

Suppose $G$ is irreducible.  Schur's lemma implies that $Z(G) \subset Z(\GL_n(\C))$.
In particular, $Z(G)$ is cyclic.  If there exists a nontrivial decomposition 
$V = W_1 \oplus \cdots \oplus W_k$ such that $G$ acts transitively on the $W_j$, then we
say $G$ is imprimitive.  In this case, each $W_j$ has the same dimension, and
the standard representation is induced from a representation on $W_1$.  Otherwise, call $G$
primitive.

Let $A \mapsto \bar A$ denote the quotient map from $\GL_n(\C)$ to $\PGL_n(\C)$.  Similarly,
if $G \subset \GL_n(\C)$, let $\bar G$ be the image of $G$ under this map.  
We call the projective image $\bar G$
irreducible or primitive if $G$ is. Finite
subgroups of $\PGL_n(\C)$ have been classified for small $n$, and we can use this to
describe the finite subgroups of $\GL_n(\C)$.  

Namely, suppose $G \subset \GL_n(\C)$ is irreducible.  Then $Z(G)$ is a cyclic subgroup
of scalar matrices, and $\bar G = G/Z(G)$.  Hence the irreducible finite subgroups of 
$\GL_n(\C)$, up to isomorphism, are a subset of the set of 
finite cyclic central extensions of the irreducible subgroups $\bar G$ of $\PGL_n(\C)$.  

Let $H$ be an irreducible subgroup of $\PGL_n(\C)$.
Given one cyclic central extension $G$ of $H$ which embeds (irreducibly) in $\GL_n(\C)$,
note that the central product $G \times_{Z(G)} C_m$ also does for any cyclic group 
$C_m \supset Z(G)$, and has the same projective image as $G$.  
(Inside $\GL_n(\C)$, this central product just corresponds to adjoining more scalar
matrices to $G$.)  Conversely, if $G \times_{Z(G)} C_m$ is an irreducible subgroup of 
$\GL_n(\C)$, so is $G$.
We say $G$ is a minimal lift  of $H$ to $\GL_n(\C)$ if $G$ is an irreducible subgroup of 
$\GL_n(\C)$ with $\bar G \simeq H$ such that $G$ is not isomorphic to
$G_0 \times_{Z(G_0)} C_m$ for any proper subgroup $G_0$ of $G$.

\subsection{Serre's construction}
\label{sec:serre}

Here we explain the proof of Proposition \ref{prop1}
and describe Serre's construction.  

Suppose $\chi_1$ and $\chi_2$ are two distinct irreducible degree $n$ characters of 
a finite group $G$.
Let $Y$ be the set of elements of $G$ such that $\chi_1(g) \ne \chi_2(g)$.  Then we have
\[ |G| ( (\chi_1, \chi_1) - (\chi_1, \chi_2) ) =
\sum_Y \chi_1(g) (\overline{\chi_1(g)} - \overline{\chi_2(g)}). \]
Using the bound $|\chi_i(g)| \le n$ for $i=1,2$ and orthogonality relations, we see
\[ |G| = |G| ( (\chi_1, \chi_1) - (\chi_1, \chi_2) )
\le 2n^2 |Y|. \]
This proves Proposition \ref{prop1}.

\medskip
We now recall Serre's construction proving Theorem \ref{thm:serre}, which 
is briefly described in \cite{ramakrishnan:motives} using observations from \cite[Sec 6.5]{serre}.

Let $H$ be an irreducible subgroup of $\GL_n(\C)$, containing $\zeta I$ for each $n$-th root
of unity $\zeta$, such that $\bar H$ has order $n^2$.  
This means that $H$ is of ``central type'' with cyclic center.  Such $H$ exist for all $n$.
For instance, one can take $\bar H = A \times A$, where $A$ is an abelian group of order 
$n$.  Some nonabelian examples of such
$\bar H$ are given by Iwahori and Matsumoto \cite[Sec 5]{iwahori-matsumoto}.
Iwahori and Matsumoto
conjectured that groups of central type are necessarily solvable and this was proved
using the classification of finite simple groups by Howlett and Isaacs \cite{howlett-isaacs}.

Since $|H| = n^3$ and $|Z(H)| = n$, the identity $\sum |\tr h|^2 = |H|$
implies $\tr h = 0$ for each $h \in H \setminus Z(H)$, i.e., 
the set of $h \in H$ such that $\tr h = 0$ has cardinality  $n^3-n = (1 - \frac 1{n^2}) |H|$. 

Let $G = H \times \{ \pm 1 \}$ and consider the representations of $G$
given by  $\rho = \tau \otimes 1$ and $\rho' = \tau \otimes
\text{sgn}$, 
where $\tau$ is the standard representation of $H$ and $\text{sgn}$
is the nontrivial character of $\{ \pm 1 \}$.  Then $\tr \rho(g) = \tr \rho'(g) = 0$
for $2(n^3-n) = (1 - \frac 1{n^2} ) |G|$ elements of $G$.  On the remaining $2n$ elements
of $Z(G)$, $\tr \rho$ and $\tr \rho'$ must differ on precisely
$n$ elements, giving $G$ with $\delta_n(G) = \frac 1{2n^2}$ as desired.

Finally that $\rho$ and $\rho'$ so constructed are induced for $n > 1$.  
It suffices to show $\tau$ is induced.
Since $\bar H$ is solvable, there is a subgroup of prime index $p$, so there exists
a subgroup $K$ of $H$ of index $p$ which contains $Z=Z(H)$.  Put $\chi = \tr \tau$. 
Now $\sum_K |\chi(k)|^2
= \sum_Z |\chi(k)|^2 = |H|$.  On the other hand $\sum_K |\chi(k)|^2 \le
\sum_{i=1}^r \sum_K |\psi_i(k)|^2 = r |K|$, where $\chi |_K = \psi_1 + \cdots + \psi_r$
is the decomposition of $\chi|_K$ into irreducible characters of $K$.  Thus $r \ge p$
and we must have equality, which means $\tau$ is induced from a $\psi_i$.
We note that, more generally, Christina Durfee informed us of a proof that $\rho$, $\rho'$ must be induced if $\delta(\rho, \rho') = \frac 1{2n^2}$.  

\section{General Methods} \label{sec:general}

\subsection{Central extensions and minimal lifts}
The first step in the proof of Theorem \ref{thm1} is the determination of the minimal
lifts of irreducible finite subgroups of $\PGL_2(\C)$ and $\PGL_3(\C)$.  Here we explain
our method for this.

Let $G$ be a group and $A$ an additive abelian group, which we view as a $G$-module with
trivial action.  Then a short exact sequence of groups
\begin{equation} \label{eq:ses}
 0 \to A \overset{\iota}{\to} H \overset{\pi}{\to} G \to 1,
\end{equation}
where $\iota$ and $\pi$ are homomorphisms, such that $\iota(A) \subset Z(H)$
gives a central extension $H$ of $G$ by $A$.  Let $M(G,A)$ be the set of such sequences.
(Note these sequences are often called central extensions, but for our purpose it makes sense to call the middle term $H$ the central extension.)
We say two sequences in $M(G,A)$ are equivalent if there is a map $\phi$ that makes
this diagram commute:
\begin{equation} \label{eq:cd}
  \raisebox{-0.5\height}{\includegraphics{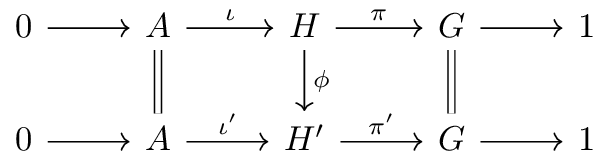}}
\end{equation}
Let $\~M(G,A)$ be $M(G,A)$ modulo equivalence.  

If two sequences in $M(G,A)$ as above are equivalent, then $H \simeq H'$.  
However the converse is not true.  E.g., taking $G\simeq A \simeq C_p$, then
$|\~M(G,A)| = p$ but there are only two isomorphism classes of central extensions
of $C_p$ by itself, namely the two abelian groups of order $p^2$.

Let $\Cent(G,A)$ be the set of isomorphism classes of central extensions of $G$ by $A$.
Then the above discussion shows we have a surjective but not necessarily
injective map $\Phi : \~M(G,A) \to \Cent(G,A)$
induced from sending a sequence as in \eqref{eq:ses} to the isomorphism class of $H$.

Viewing $A$ as a trivial $G$-module, we have a bijection between $\~M(G,A)$ and 
$H^2(G,A)$, with the class $0 \in H^2(G,A)$ corresponding to all split sequences in 
$M(G,A)$. 
We can use this to help determine minimal lifts of irreducible subgroups of 
$\PGL_n(\C)$.  We recall $H_1(G,\Z)$ is the abelianization of $G$, and $H_2(G,\Z)$
is the Schur multiplier of $G$.

\begin{prop} \label{prop:31}
 Let $G$ be an irreducible subgroup of $\PGL_n(\C)$.
Then any minimal lift of $G$ to $\GL_n(\C)$ is a central extension of $G$ by $C_m$ 
for some divisor $m$ of the exponent of  $H_1(G,\Z) \times H_2(G,\Z)$.
\end{prop}

\begin{proof} Any lift of $G$ to an irreducible subgroup $H \subset \GL_n(\C)$ corresponds
to an element of $\Cent(G,A)$ where $A=C_m$ for some $m$, and thus corresponds to
at least one element of $H^2(G,A)$.
The universal coefficients theorem gives us the exact sequence
\begin{equation} \label{eq:uct}
 0 \to \mathrm{Ext}(H_1(G,\Z), A) \to H^2(G, A) \to \mathrm{Hom}(H_2(G,\Z), A) \to 0.
\end{equation}
Let $m'$ be the gcd of $m$ with the exponent of $H_1(G,\Z) \times H_2(G,\Z)$. 
Recall that $\mathrm{Ext}(\bigoplus \Z/n_i \Z, A) = \bigoplus A/n_i A$, 
so $\mathrm{Ext}(H_1(G,\Z), C_m) = \mathrm{Ext}(H_1(G,\Z), C_{m'})$.  
An analogous statement is true for $\mathrm{Hom}(H_2(G,\Z), -)$ so 
$|H^2(G, C_m)| = |H^2(G, C_{m'})|$.  

Assume $m \ne m'$.  Consider a sequence as in \eqref{eq:ses} with $A = C_{m'}$.
This gives a sequence
\[ 0 \to C_m \to H \times_{C_{m'}} C_m \to G \to 1 \]
in $M(G,C_m)$ by extending $\iota : C_{m'} \to H$ to be the identity on $C_m$.  
Note if one has an equivalence $\phi$ of two sequences in $M(G, C_m)$ constructed
in this way, then commutativity implies $\phi(H) = H$ so restricting
the isomorphism $\phi$ on the middle groups to $H$ yields and equivalence of the corresponding
sequences in $M(G, C_{m'})$.  Hence all elements of $\~M(G, C_m)$ arise from 
``central products'' of sequences in $M(G, C_{m'})$, and thus no elements of 
$\Cent(G, C_m)$ can be minimal lifts.
\end{proof}

When $H_1(G,\Z) \times H_2(G, \Z) \simeq 1$, then $H^2(G,A) = 0$ for any abelian group
$A$, which means all central extensions are split, i.e., $\Cent(G, A) = \{ G \times A \}$ for
any $A$.  When $H_1(G,\Z) \times H_2(G, \Z) \simeq \Z/2\Z$, then \eqref{eq:uct}
tells us that $|H^2(G, C_m)|$ has size 1 or 2 according to whether $m$ is odd
or even, so there must be a unique nonsplit extension $\~ G \in \Cent(G,C_2)$.  Then
the argument in the proof tells us any cyclic central extension of $G$ is a central product
of either $G$ or $\~G$ with a cyclic group.
 
However, in general, knowing $H_1(G,\Z)$ and $H_2(G, \Z)$ is not enough to determine the
size of $\Cent(G, C_m)$.  When $|\Cent(G, C_m)| < |H^2(G, C_m)|$, we will sometimes
need a way to verify that the central extensions of $G$ by $C_m$ we exhibit exhaust all
of $\Cent(G, C_m)$.  For this, we will use a lower bound on the size of the fibers of $\Phi$, 
i.e., a lower bound on the number of classes in $\~M(G, A)$ a given central extension $H \in \Cent(G,A)$ appears in.

The central automorphisms of a group $H$ with center $Z$, denoted $\Aut_Z(H)$,
are the automorphisms $\sigma$ of $H$ which commute with the projection $H \to H/Z$,
i.e., satisfy $\sigma(h)h^{-1} \in Z$ for all $h \in H$.

\begin{prop}  \label{prop:32}
Let $A$ be abelian and $H \in \Cent(G,A)$ such that
$A = Z:= Z(H)$.  Then 
$|\Phi^{-1}(H)| \ge \frac{|\Aut(Z)|}{|\Aut_Z(H)|}$.  Moreover, if $H$ is perfect, then
$|\Phi^{-1}(H)| \ge |\Aut(Z)|$.
\end{prop}

Recall $H$ being perfect means $H$ equals its derived group, i.e., $H_1(H,\Z) = 0$.
In particular, non-abelian simple groups are perfect.
By \eqref{eq:uct}, central extensions of perfect groups are simpler to study. 
In fact a perfect group $H$ possesses a universal central extension by $H_2(H,\Z)$.

\begin{proof}
Consider a commuting diagram of sequences as in
\eqref{eq:cd} with $H'=H$.  Suppose $\pi = \pi'$, which forces $\phi \in \Aut_Z(H)$
and $\iota'(A) = \ker \pi =  \iota(A)$.  Fixing $\pi$ and $\iota$, 
there are $|\Aut(Z)|$ choices for $\iota'$, which gives $|\Aut(Z)|$ elements of $M(G,A)$.
Each different $\iota'$ must induce a different central
automorphism $\phi \in \Aut_Z(H)$.  Thus at most $|\Aut_Z(H)|$ of these 
$|\Aut(Z)|$ bottom sequences can lie in the same
equivalence class, which proves the first statement.

Adney and Yen \cite{adney-yen} showed $|\Aut_Z(H)| = |\mathrm{Hom}(H,Z)|$
when $H$ has no abelian direct factor.  Consequently, $\Aut_Z(H) = 1$ when $H$ is perfect.
\end{proof}

\subsection{Reduction to minimal lifts} \label{sec:reduction}

Let $G$ be a finite group and $\rho_1, \rho_2$ be two inequivalent irreducible
representations of $G$ into $\GL_n(\C)$.  Let $N_i = \ker \rho_i$ and $G_i = \rho_i(G)$
for $i = 1, 2$.  We want to reduce the problem of finding lower bounds for $\delta(\rho_1,
\rho_2)$ to the case where $G_1$ and $G_2$ are minimal lifts of $\bar G_1$ and 
$\bar G_2$.  Note that $\delta(\rho_1, \rho_2)$ is unchanged if we factor through the 
common kernel $N_1 \cap N_2$, so we may assume $N_1 \cap N_2 = 1$.
Then $N_1 \times N_2$ is a normal subgroup of $G$, $N_1 \simeq \rho_2(N_1) \lhd G_2$
and $N_2 \simeq \rho_1(N_2) \lhd G_1$.

Write $G_i = H_i \times_{Z(H_i)} Z_i$ for $i=1, 2$, where $H_i$ is a minimal lift of
$\bar G_i$ to $\GL_n(\C)$ and $Z_i$ is a cyclic group containing $Z(H_i)$.

For a subgroup $H$ of $\GL_n(\C)$, let $\alpha_n(H)$ be the minimum of
$\frac{ | \{ h \in H : \tr h \ne 0 \} | }{|H|}$ as one ranges over all embeddings (i.e., faithful
$n$-dimensional representations) of $H$ in $\GL_n(\C)$.

\begin{lemma} \label{lem:34}
Let $m = |\rho_1(N_2) \cap Z(G_1)|$.  
Then $\delta(\rho_1, \rho_2) \ge \frac{m-1}{m} \alpha_n(H_1)$. 
\end{lemma}

\begin{proof} Let $K = N_2 \cap \rho_1^{-1}(Z(G_1))$, so $\rho_1(K)$ is a cyclic subgroup
of $Z(G_1)$ of order $m$ and $\rho_2(K) = 1$.  Fix any $g \in G$.  Then as $k$ ranges over
$K$, $\tr \rho_1(gk)$ ranges over the values $\zeta \tr \rho(g)$, where $\zeta$ runs through
all $m$-th roots of 1 in $\C$, attaining each value equally often.  On the other hand,
$\tr \rho_2(gk) = \tr \rho_2(g)$ for all $k \in K$.  So provided $\tr \rho_1(g) \ne 0$,
$\tr \rho_1$ and $\tr \rho_2$ can agree on at most $\frac{1}m |K|$ values on the coset $gK$.
Then note that the fraction of elements $g \in G$ for which $\tr \rho_1(g) \ne 0$ is the
same as the fraction of elements in $h \in H_1$ for which $\tr h \ne 0$.
\end{proof}

We say a subgroup $H_0$ of a group $H$ is $Z(H)$-free if $H_0 \ne 1$ and
$H_0 \cap Z(H) = 1$.
The above lemma implies that if $G_1$ has no $Z(G_1)$-free normal subgroups, 
then $\delta(\rho_1, \rho_2) \ge \frac {\alpha_n(H_1)} 2$ or $N_2 = 1$ (as the $K$ in the proof
must be nontrivial).
This will often allow us to reduce to the case where $N_2 = 1$, and similarly $N_1 = 1$, i.e., 
$G = G_1 = G_2$, when we can check this property for $G_1$ and $G_2$.
The following allows us to simply check it for $H_1$ and $H_2$.

\begin{lemma} If $H_1$ has no $Z(H_1)$-free normal subgroups, then
$G_1$ has no $Z(G_1)$-free normal subgroups.
\end{lemma}

\begin{proof}
Suppose $H_1$ has no $Z(H_1)$-free normal subgroups, but that $N$ is a $Z(G_1)$-free
normal subgroup of $G_1$.  Let $N' = \{ n \in H_1 : (n, z) \in G_1 = H_1 \times_{Z(H_1)} Z_1
\text{ for some } z \in Z_1 \}$.  Then $N' \lhd H_1$.  If $N' = 1$, then 
$N \subset Z_1 = Z(G_1)$, contradicting $N$ being $Z(G_1)$-free.  Hence $N' \ne 1$ and
must contain a nontrivial $a \in Z(H_1)$.   But then $(a,z) \in N \cap Z(G_1)$ for some
$z \in Z_1$, which also contradicts $N$ being $Z(G_1)$-free.
\end{proof}

This will often allow us to reduce to the case where $G = H \times_{Z(H)} A$ for some
cyclic group $A \supset Z(H)$, where we can use the following.

\begin{lemma} \label{lem:35}
 Let $H$ be a finite group, $A \supset Z(H)$ an abelian group and
$G = H \times_{Z(H)} A$.  Then $\delta^\natural_n(G) \ge \min 
\{ \frac 12 \alpha_n(H), \delta^\natural_n(H) \}$.
\end{lemma}

\begin{proof}
We may assume $m = |A| > 1$.  Let $\rho_1, \rho_2: G \to \GL_n(\C)$ be distinct primitive representations of $G$.  They pull back to $H \times A$, so for $i=1,2$ 
we can view $\rho_i =  \tau_i \otimes \chi_i$ where $\tau_i: H \to \GL_n(\C)$ is primitive and
$\chi_i: A \to \C^\times$.  By a similar argument to the proof of Lemma \ref{lem:34}, we have
that $\delta(\rho_1, \rho_2) \ge \frac{m-1}m \alpha_n(H)$ if 
$\chi_1 \ne \chi_2$.  If $\chi_1 = \chi_2$, it is easy to see $\delta(\rho_1, \rho_2) = \delta(\tau_1, \tau_2)$.
\end{proof}

In the simplest situation, this method gives the following.

\begin{cor} \label{cor:method}
Let $\mathcal H$ be the set of minimal lifts of $\bar G_1$ and $\bar G_2$ to
$\GL_n(\C)$.   Suppose that $H$ has no $Z(H)$-free normal subgroups for all $H \in \mathcal H$.  
Then 
\[ \delta(\rho_1, \rho_2) \ge \min  \{ \frac 12 \alpha_n(H), \delta^\natural_n(H) : H \in \mathcal H \}. \]
\end{cor}

This corollary will address most but not all cases of our proof of Theorem \ref{thm1}.  
Namely, when $n=3$,
it can happen that $\bar G_1$ has a lift $H \simeq \bar G_1$ which is simple, 
so $H$ is a $Z(H)$-free normal subgroup of itself.  So we will need to augment this
approach when $H_1$ or $H_2$ is simple. 

\section{Primitive degree 2 characters}

In this section we will prove the $n=2$ case of Theorem \ref{thm1}.

We used the computer package GAP 4 \cite{gap} for explicit group and character 
calculations in this section and the next.  
We use the notation $[n, m]$ for the $m$-th group of order $n$ in the Small Groups Library,
which is accessible by the command \verb+SmallGroup(n,m)+ in GAP.  
We can enumerate all (central or not) extensions of $G$ by $N$ in GAP if 
$|G||N| \le 2000$ and $|G||N| \ne 1024$ as all groups of these orders are in the Small 
Groups Library.  We can also compute homology groups $H_n(G,\Z)$ using the
HAP package in GAP.

\subsection{Finite subgroups of $\GL_2(\C)$}

Recall the classification of finite subgroups of $\PGL_2(\C) \simeq \mathrm{SO}_3(\C)$.
Any finite subgroup of $\PGL_2(\C)$ is of one of the following types:

\begin{enumerate}[(A)]
\item cyclic

\item dihedral

\item tetrahedral ($A_4 \simeq \PSL(2,3)$)

\item octahedral ($S_4$)

\item icosahedral ($A_5  \simeq \PSL(2,5) \simeq \PSL(2,4) \simeq \SL(2,4)$)
\end{enumerate}

Now suppose $G$ is a subgroup of $\GL_2(\C)$ with projective image $\bar G$ in
$\PGL_2(\C)$.  If $\bar G$ is cyclic, $G$ is reducible.  If $\bar G$
is dihedral, then $G$ is not primitive.

Assume $\bar G$ is primitive.  Then we have the following possibilities.

\begin{enumerate}[(A)]
\setcounter{enumi}{2}
\item  Suppose $\bar G = A_4 \simeq \PSL(2,3)$.  Here $H_1(A_4, \Z) \simeq \Z/3\Z$ and
$H_2(A_4, \Z) \simeq \Z/2\Z$.  There is one nonsplit element of $\Cent(A_4, C_2)$, namely
$\SL(2,3)$; one nonsplit element of $\Cent(A_4, C_3)$, namely [36, 3]; and one
element of $\Cent(A_4, C_6)$ which is not a central product with a smaller extension,
namely [72, 3].  Of these central extensions (and the trivial extension $A_4$), 
only $\SL(2,3)$ and $[72, 3]$ have irreducible
faithful 2-dimensional representations.  

Thus $\SL(2,3)$ and $[72, 3]$ are the only minimal lifts of $A_4$ to $\GL_2(\C)$.
We check that neither $H=\SL(2,3)$ nor $H=[72, 3]$ has $Z(H)$-free normal subgroups.
In both cases, we have $\alpha_2(H) = \frac 34$, and 
 $\delta_2^\natural(H) = \frac 23$.

\item Next suppose $\bar G = S_4$.  
Note $H_1(S_4, \Z) \simeq H_2(S_4, \Z) \simeq \Z/2\Z$.  There are 3 nonsplit central
extensions of $S_4$ by $C_2$: [48, 28], [48, 29], [48, 30].  Neither $S_4$ nor [48, 30]
have faithful irreducible 2-dimensional representations, but both [48, 28] and [48, 29] do.

Thus $H=[48, 28]$ and $H=[48, 29]$ are the minimal lifts of $S_4$ to $\GL_2(\C)$.  Neither
of them have $Z(H)$-free normal subgroups.  In both cases we compute
$\alpha_2(H) = \frac 58$ and $\delta_2^\natural(H) = \frac 14$.

\item Last, suppose $\bar G = A_5 = \PSL(2,5)$.  This group is perfect and 
$H_2(A_5, \Z) \simeq \Z/2\Z$, with $\SL(2,5)$ being the nontrivial central extension by
$C_2$ (the universal central extension).  Note $A_5$ has no irreducible 2-dimensional representations.

Hence there is only one minimal lift of $A_5$ to $\GL_2(\C)$, $H=\SL(2,5)$.
We can check that $\SL(2,5)$ has no $Z(\SL(2,5))$-free normal subgroups,
$\alpha_2(\SL(2,5)) = \frac 34$ and $\delta_2^\natural(\SL(2,5)) = \frac 25$  (cf.\ Theorem
\ref{thm2}).
\end{enumerate}

\subsection{Comparing characters}

Let $\rho_1, \rho_2: G \to \GL_2(\C)$ be inequivalent primitive representations.
By Corollary \ref{cor:method}, 
\[ \delta(\rho_1, \rho_2) \ge \min \left( \frac 12 \left\{ \frac 34, \frac 58, \frac 34 \right\} \cup 
\left\{ \frac 23, \frac 14, \frac 25 \right\} \right) = \frac 14. \]
This shows $d_2^\natural \ge \frac 14$.  Furthermore, we can only have $\delta(\rho_1, \rho_2) = \frac 14$ if $\bar G_1$ or $\bar G_2$ is $S_4$, which implies $G_1$ or 
$G_2$ is of the form $H \times_{C_2} C_{2m}$ for some $m$ with $H=[48, 28]$ or 
$H=[48, 29]$.  Thus we can only have $\delta_2^\natural(G) = \frac 14$ if $G$
is an extension of $H \times_{C_2} C_{2m}$ where $m \in \mathbb N$ and $H=[48, 28]$ or $H=[48, 29]$.  Moreover, if $G$ is such an extension
$\delta_2^\natural(G)$ equals $\frac 14$ because $\delta_2^\natural(H)$ does.

This completes the proof of Theorem \ref{thm1} when $n=2$.

\section{Primitive degree 3 characters}

Here we prove the $n=3$ case of Theorem \ref{thm1}.

\subsection{Finite subgroups of $\GL_3(\C)$}

First we review the classification of finite subgroups $\GL_3(\C)$.
The classification can be found in Blichfeldt \cite{blichfeldt} or Miller--Blichfeldt--Dickson
\cite{MBD}.  We follow the classification system therein.  The description involves
3 not-well-known groups, $G_{36} = [36,9]$, $G_{72} = [72, 41]$, and $G_{216} = [216,153]$.  
Explicit matrix presentations for preimages in $\GL_3(\C)$ are
given in \cite[Sec 8.1]{me:thesis}.

Any finite subgroup $G$ of $\GL_3(\C)$ with projective image $\bar G$ 
is one of the following types, up to conjugacy:

\begin{enumerate}[(A)]
\item abelian

\item a nonabelian subgroup of $\GL_1(\C) \times \GL_2(\C)$

\item a group generated by a diagonal subgroup and $\bmx & 1 & \\ && 1 \\ 1 & & \\ \emx$ 

\item a group generated by a diagonal subgroup, $\bmx & 1 & \\ && 1 \\ 1 & & \\ \emx$
and a matrix of the form $\bmx a && \\ && b \\ &c & \emx$

\item $\bar G \simeq G_{36}$

\item $\bar G \simeq G_{72}$

\item $\bar G \simeq G_{216}$

\item $\bar G \simeq A_5 \simeq \PSL(2,5) \simeq \PSL(2,4) \simeq \SL(2,4)$

\item $\bar G \simeq A_6 \simeq \PSL(2,9)$

\item $\bar G \simeq \PSL(2,7)$
\end{enumerate}

Of these types, (A), (B) are reducible, (C), (D) are imprimitive, and the remaining types
are primitive.  The first 3 primitive groups, (E), (F) and (G),  have non-simple projective images, whereas the latter 3, (H), (I) and (J), have simple projective images.

Now we describe the minimal lifts to $\GL_3(\C)$ of $\bar G$ for cases (E)--(J).

\begin{enumerate}[(A)]
\setcounter{enumi}{4}
\item We have $H_1(G_{36}, \Z) \simeq \Z/4\Z$ and $H_2(G_{36}, \Z) \simeq
\Z/3\Z$.  The nonsplit extension of $G_{36}$ by $C_2$ is [72, 19].  There is 
one non split extension of $G_{36}$ by $C_4$ which is not a central product,
[144, 51].  However, $G_{36}$, [72,~19] and [144,~51] all have no irreducible
3-dimensional representations.

There is 1 nonsplit central extension of $G_{36}$ by $C_3$, [108, 15];
there is one by $C_6$ which is not a central product, [216, 25]; there is one by
$C_{12}$ which is not a central product, [432, 57].  All of these groups have
faithful irreducible 3-dimensional representations.

Hence any minimal lift of $G_{36}$ to $\GL_3(\C)$ is $H=[108, 15]$, 
$H=[216, 25]$ or $H=[432, 57]$.  In all of these cases, $H$ has no $Z(H)$-free
normal subgroups, $\alpha_3(H) = \frac 79$ and $\delta_3^\natural(H) = \frac 12$.

\item We have $H_1(G_{72}, \Z) \simeq \Z/2\Z \times \Z/2\Z$ and 
$H_2(G_{72}, \Z) \simeq \Z/3\Z$.  There is a unique nonsplit central extension of $G_{36}$ by $C_2$, [144, 120]; a unique central extension of $G$ by $C_3$,
[216, 88]; and a unique central extension of $G$ by $C_6$ which is not a central
product, [432, 239].   Of these extensions (including $G_{72}$), only the latter two 
groups have faithful irreducible 3-dimensional representations.

Thus there are two minimal lifts of $G_{72}$ to $\GL_3(\C)$, 
$H=[216, 88]$ and $H=[432, 239]$.  In both cases, $H$ has no $Z(H)$-free
normal subgroups, $\alpha_3(H) = \frac 89$ and $\delta_3^\natural(H) = \frac 12$.

\item We have $H_1(G_{216}, \Z) \simeq H_2(G_{216}, \Z) \simeq \Z/3\Z$.
There are 4 nonsplit central extensions of $G_{216}$ by $C_3$: [648, 531],
[648, 532], [648, 533], and [648, 534].  Neither $G_{216}$ nor [648, 534] has
irreducible faithful 3-dimensional representations.

Thus there are three minimal lifts of $G_{216}$ to $\GL_3(\C)$,
$H=[648, 531]$, $H=[648, 532]$, and $H=[648, 533]$.  In all cases 
$H$ has no $Z(H)$-free
normal subgroups, $\alpha_3(H) = \frac {20}{27}$ and $\delta_3^\natural(H) = \frac 49$.

\item As mentioned in the $n=2$ case, $A_5 \simeq \PSL(2, 5)$ is perfect and we have 
$H_2(A_5, \Z) \simeq \Z/2\Z$.
The nontrivial extension by $C_2$ (the universal central extension) is $\SL(2,5)$,
but $\SL(2,5)$ has no faithful irreducible 3-dimensional representations.

Thus the only minimal lift of $A_5$ to $\GL_3(\C)$ is $A_5$ itself.  We have
$\alpha_3(A_5) = \frac 23$ and $\delta_3^\natural(A_5) = \delta_3^\natural(\PSL(2,5)) =
\frac 25$ (cf.\ Theorem \ref{thm2}).

\item The group $A_6$ is also perfect, but (along with $A_7$) exceptional among 
alternating groups in that $H_2(A_6, \Z) \simeq \Z/6 \Z$.  Neither $A_6 \simeq \PSL(2,9)$, nor its double cover $\SL(2,9)$, has
irreducible 3-dimensional representations. There is a unique nonsplit central 
extension of $A_6$ by $C_3$, sometimes called the Valentiner group, which
we denote $V_{1080} = [1080, 260]$ and is also a perfect group.  It is known (by
Valentiner) that $V_{1080}$ has an irreducible faithful 3-dimensional representation.

To complete the determination of minimal lifts of $A_6$ to $\GL_3(\C)$, we need
to determine the central extensions of $A_6$ by $C_6$.  Here we cannot (easily) 
proceed naively as in the other cases of testing all groups of the appropriate order 
because we do not have a library of all groups of order 2160.  We have
$|\~M(A_6, C_6)| = 6$, with one class accounted for by the split extension and
one by $\SL(2,9) \times_{C_2} C_6$.
Since $V_{1080}$ must correspond to two classes in $\~M(A_6, C_3)$, 
$V_{1080} \times_{C_3} C_6$ corresponds to two classes in $\~M(A_6, C_6)$
by the proof of Proposition \ref{prop:31}.  Since $A_6$ is perfect, it has a universal
central extension by $C_6$, which we denote $\~A_6$.  By Proposition \ref{prop:32}, 
$\~A_6$ accounts for the remaining 2 classes in $\~M(A_6, C_6)$, and thus we have 
described all elements of $\Cent(A_6, C_6)$.  The group $\~A_6$
is the unique perfect group of order 2160 and can be accessed by the command 
\verb+PerfectGroup(2160)+ in GAP, and we can check that it has no faithful 
irreducible 3-dimensional representations.

Hence $V_{1080}$ is the unique minimal lift of $A_6$ to $\GL_3(\C)$.  We note
$H=V_{1080}$ has no $Z(H)$-free normal subgroups, $\alpha_3(H) = \frac 79$,
and $\delta_3^\natural(H) = \frac 25$.

\item The group $\PSL(2,7)$ is perfect and $H_2(\PSL(2,7), \Z) \simeq \Z/2\Z$.
Since $\SL(2,7)$ has no faithful irreducible 3-dimensional representations, any
minimal lift of $\PSL(2,7)$ to $\GL_3(\C)$ is just $H = \PSL(2,7)$.  
Here $\alpha_3(H) = \frac 23$ and $\delta_3^\natural(H) = \frac 27$ by
Theorem \ref{thm2}.
\end{enumerate}

\begin{table}
\begin{center}
{\renewcommand{\arraystretch}{1.2}
\begin{tabular}{L|CCCCCC}
\bar G  & 3 & 1 & 0 & \frac{1\pm \sqrt 5}2 & \sqrt 2 & \sqrt 3\\
\hline
G_{36} & \frac{1}{36} & \frac 34 & \frac 29 & & &\\
G_{72} & \frac 1{72} & \frac 78 & \frac 1{9} && & \\
G_{216}  & \frac 1{216} & \frac 58 & \frac 7{27} & & & \frac 19 \\
A_5& \frac 1{60} & \frac 14 & \frac 13 & \frac 25 &\\
A_6  & \frac 1{360} & \frac 38 & \frac 29 & \frac 25 &\\
\PSL(2,7) & \frac 1{168} & \frac 38 & \frac 13 & & \frac 27 &
\end{tabular}
}
\caption{Fraction of group elements with primitive degree 3 characters having given absolute value}
\label{tabby}
\end{center}
\end{table}

\subsection{Comparing characters}

Let $G$ be a finite group and $\rho_1, \rho_2: G \to \GL_3(\C)$ be two inequivalent
primitive representations.  Let $G_i, N_i, H_i, Z_i$ be as in Section \ref{sec:reduction}.  As before, we may assume $N_1 \cap N_2 = 1$, so
$G$ contains a normal subgroup isomorphic to $N_1 \times N_2$ whose
image in $G_1$ is $N_2$ and image in $G_2$ is $N_1$.

\begin{prop} \label{prop:deg3}
Suppose at least one of $\bar G_1$, $\bar G_2$ is simple.  Then
$\delta(\rho_1, \rho_2) \ge \frac 27$, with equality only if $\bar G_1 \simeq \bar G_2 \simeq
\PSL(2,7)$.
\end{prop}

\begin{proof}
Say $\bar G_1$ is simple.  Then by above, $H_1$ is isomorphic to one of 
$A_5$, $V_{1080}$ and $\PSL(2,7)$.  

\medskip
{\bf Case I:} Suppose $\bar G_1 \not \simeq \bar G_2$.
For $i=1,2$, the fraction of $g\in G$ for which $|\tr \rho_i(g)| = x$ is the same as
the fraction of $h \in H_i$ for which $|\tr h| = x$.  Calculations show
that the proportion of such $g \in G$ (given $x$) depends neither on the minimal
lift $H_i$ nor its embedding into $\GL_3(\C)$, but just on $\bar G_i$.  These proportions are given in Table \ref{tabby}.

If $\bar G_1 \simeq \PSL(2,7)$, we see $\delta(\rho_1, \rho_2) \ge \frac 27$
just from considering elements with absolute character value $\sqrt 2$.  Looking
at other absolute character values shows this inequality is strict.

If $\bar G_1 \simeq A_5$ or $A_6$ and $\bar G_2$ is not isomorphic to $A_5$ or
$A_6$, then considering elements with absolute character value 
$\frac{1\pm \sqrt 5}2$ shows $\delta(\rho_1, \rho_2) \ge \frac 25$.

So assume $\bar G_1 \simeq A_5$ and $\bar G_2 \simeq A_6$.  Then
$G_1 = A_5 \times C_m$ and $G_2 \simeq V_{1080} \times_{C_3} C_{3r}$ for some
$m, r \in \mathbb N$.  Suppose $\delta(\rho_1, \rho_2) < \frac 13$.  By 
Lemma \ref{lem:34}, $\rho_1(N_2)$ and $\rho_2(N_1)$ are either $Z(G_1)$- and
$Z(G_2)$-free normal subgroups of $G_1$ and $G_2$ or trivial.  This forces
$N_1 = 1$, so $G \simeq G_1$, but it is impossible for a quotient of $G_1$ to
be isomorphic to $G_2$.  Hence $\delta(\rho_1, \rho_2) \ge \frac 13 > \frac 27$ in
this case.

\medskip
{\bf Case II:} Suppose $\bar G_1 \simeq \bar G_2$.

First suppose $N_1$ or $N_2$ is trivial, say $N_1$.  Then $G \simeq G_1$. 
By Lemma \ref{lem:35}, we have 
$\delta_3^\natural(G) \ge \min \{ \frac 13, \delta_3^\natural(H_1) \}$.  Thus
$\delta_3^\natural(G) = \frac 27$ if and only if $H_1 = \PSL(2,7)$.

So assume $N_1$ and $N_2$ are nontrivial.  By Lemma \ref{lem:34},
we can assume $\rho_1(N_2)$ and $\rho(N_1)$ are 
$Z(G_1)$- and $Z(G_2)$-free normal subgroups of $G_1$ and $G_2$.  This is only
possible if $N_1 \simeq N_2 \simeq H_1 \simeq H_2$ is isomorphic to 
$A_5$ or $\PSL(2,7)$.  

Let $N = \rho_1^{-1}(N_2) \lhd G$ and we identify $N = N_1 \times N_2$.
Fix $g \in G$.  Then for any $n_1 \in N_1$,  
$\tr \rho_1(g(n_1, 1)) = \tr \rho_1(g)$ but $\tr \rho_2(g(n_1,1)) = \tr \rho_2(g(n_1,1))$.
Since $\rho_2( g(N_1 \times 1)) = H_2 \times \{ z \}$ for some $z \in Z_2$, 
the fraction of elements of $g (N_1 \times 1)$ (and thus of $G$) on which
$\tr \rho_1$ and $\tr \rho_2$ can agree is at most the maximal fraction of elements
of $H_1$ with a given trace.  By Table \ref{tabby} this is less than $\frac 12$ for
either $\bar G_1 \simeq A_5$ or $\bar G_1 \simeq \PSL(2,7)$.
\end{proof}

To complete the proof of Theorem \ref{thm1} for $n=3$, it suffices to show
$\delta(\rho_1, \rho_2) > \frac 27$ when $\bar G_1$ and $\bar G_2$ are each
one of $G_{36}$, $G_{72}$ and $G_{216}$. 
Using Corollary \ref{cor:method}, in this situation we see
\[ \delta(\rho_1, \rho_2) \ge \min \left( \frac 12 \left\{ \frac 79, \frac 89,
\frac {20}{27} \right\} \cup \left\{ \frac 12, \frac 12, \frac 49 \right \} \right) = \frac {10}{27}. \]
This finishes Theorem \ref{thm1}.

\section{Families $\SL(2,q)$ and $\PSL(2,q)$}
\label{sec:sl2q}

We consider ${\SL}(2,q)$ and ${\PSL}(2,q)$, for even and odd prime powers $q$. 
We separate these into three subsections: $\SL(2,q)$, $q$ odd; ${\SL}(2,q) \simeq {\PSL}(2,q) $, $q$ even; and ${\PSL}(2,q)$, $q$ odd. 
We refer to, and mostly follow the notation of, Fulton--Harris~\cite{fulton-harris} for the representations of these groups.

Choose an element $\Delta \in \F_q^\times  - (\F_q^\times )^2$. Denote by $\E:=\F_q (\sqrt{\Delta})$ the unique quadratic extension of $\F_q$. We can write the elements of $\E$ as $a + b \delta$, where $\delta := \sqrt \Delta$. The norm map $N: \E^\times  \rightarrow \F_q^\times $ is then defined as $N (a + b \delta) = a^2 - b^2  \Delta$.
We also denote $\E^1$ to be the kernel of the norm map.

\subsection{$\SL(2,q)$, for odd $q$}

The order of $\SL(2,q)$ is $(q + 1)q (q-1)$. We begin by describing the conjugacy classes for $\SL(2,q)$:

 \begin{enumerate}[(A)]
\item $I$.
\item $-I$.
\item Conjugacy classes of the form $[c_2(\epsilon,\gamma)]$, where
$c_2 (\epsilon, \gamma) =      \left( \begin{array}{cc}
\epsilon&\gamma \\
&\epsilon
\end{array} \right),$
where $\epsilon = \pm 1$ and $\gamma = 1 \text{ or } \Delta$. So there are four conjugacy classes, each of size $(q ^2 -1)/2$.\\
\item Conjugacy classes of the form $[c_3(x)]$, where
$c_3 (x) =      \left( \begin{array}{cc}
x& \\
&x ^{-1}
\end{array} \right)$
with $x \neq \pm 1$. Since the conjugacy classes $c_3 (x)$ and $c_3 (x ^{-1})$ are the same, we have $(q-3)/2$ different conjugacy classes, each of size $q (q + 1)$.\\
\item Conjugacy classes of the form $[c_4(z)]$, where
$c_4 (z) =      \left( \begin{array}{cc}
x&\Delta y \\
y&x
\end{array} \right)$
where $z = x + \delta y \in \E^1$ and $z \neq \pm 1$. Since $c_4 (z) = c_4 (\bar{z})$ we have $(q-1)/2$ conjugacy classes, each of size $q (q-1)$.\\
\end{enumerate}

We give a brief description of the representations that appear in the character table.
The first set of representations, denoted $W_\alpha$, are induced from the subgroup $B$ of upper triangular matrices. Given a character $\alpha \in \widehat{\F_q^\times}$, we can extend this to a character of $B$, which we then induce to a $(q + 1)$-dimensional representation $W_\alpha$ of ${\rm SL}_2(\F_q)$. If $\alpha^2 \neq 1$, then the induced representation is irreducible. If $\alpha = 1$, then $W_1$ decomposes into its irreducible consituents: the trivial representation $U$ and the Steinberg representation $V$. If $\alpha ^2 = 1$ and $\alpha \ne 1$, then it decomposes into two irreducible constituents denoted $W^+$ and $W^-$.

For the remaining irreducible representations, we consider characters $\alpha$ and $\varphi$ of the diagonal subgroup $A$ and the subgroup $S:=\{c_4 (z) \mid z \in \E^1\}$, respectively,  where the characters agree when restricted to $A \cap S$. Then we construct a virtual character $\pi_\varphi := {\rm Ind}^G_A (\alpha) - W_\alpha - {\rm Ind}^G_S (\varphi)$ (note that the virtual character will not depend on the specific choice of $\alpha$).

When $\varphi = \overline{\varphi}$, $\pi_\varphi$ decomposes into two distinct characters. In the case when $\varphi$ is trivial, $\pi_1$ decomposes into the difference between the characters for the Steinberg representation and the trivial representation. If $\varphi$ is the unique (non-trivial) order 2 character of $S$, then $\pi_\varphi$ decomposes into two distinct irreducible characters of equal dimension; we will label the corresponding representations $X^+$ and $X^-$.
If $\varphi \neq \overline{\varphi}$, then $\pi_\varphi$ corresponds to an irreducible representation, which we denote as $X_\varphi$.  Two irreducibles $X_\varphi$ and $X_{\varphi'}$ are equivalent  if and only if $\varphi = \varphi'$ or $\varphi = \overline{\varphi'}$.
We note that out of all the irreducible representations, the imprimitive representations are exactly all the $W_\alpha$ (for $\alpha^2 \neq 1$).

We define some notation that will appear in the character table for $\SL(2,q)$.
Let $\alpha \in \widehat{\F_q^\times}$ with $\alpha \neq \pm 1$, and $\varphi$ a character of $\E^1$ with $\varphi ^2 \neq 1$. Fix $\tau$ to be the non-trivial element of $\widehat{\F_q^\times / (\F_q^\times)^2 }$, and let 
\begin{align*}
s^\pm (\epsilon,\gamma) &= \frac{1}{2} (\tau (\epsilon) \pm \tau (\epsilon \gamma) \sqrt{\tau (-1)q}),\\
u^\pm (\epsilon,\gamma) &= \frac{1}{2} \epsilon (-\tau (\epsilon) \pm \tau (\epsilon \gamma) \sqrt{\tau (-1)q}).
\end{align*}
Lastly, we define $\psi$ to be the non-trivial element of $\widehat{\E^1 / (\E^1)^2}$.
The character table is:
    \begin{center}
      \begin{tabular}{r|c|c|c|c|c|c} 
 &   &$[I]$&$[-I]$&$[c_2(\epsilon, \gamma)]$&$[c_3(x)]$& $[c_4(z)]$ \\ \hline
 & Size: &1 &1 & $\frac{q ^2 -1}{2}$ & $q (q + 1)$& $q (q-1)$\\ \hline 
Rep & \# &&&&& \\ \hline
$U$ & 1 & 1 &1&1&1&1 \\
$X^\pm $& 2 & $\frac{q-1}{2}$ & $\frac{q-1}{2} \cdot \psi (-1)$& $u^\pm (\epsilon,\gamma)$& 0& $- \psi (z)$  \\
$W^\pm$ & 2 & $\frac{q + 1}{2}$ & $\frac{q + 1}{2} \cdot \tau (-1)$& $s^\pm (\epsilon,\gamma)$& $\tau (x)$& 0 \\
$X_\varphi$ & $\frac{q-1}{2}$ & $q-1$& $(q-1)\varphi (-1)$& $-\varphi (\epsilon)$&0& $-\varphi (z) - \varphi (z ^{-1})$ \\
$V$ & 1 & $q$ & $q$ & 0 & 1 & $-1$\\
$W_\alpha$ & $\frac{q-3}{2}$&$q + 1$&$(q + 1)\alpha (-1)$&$\alpha (\epsilon)$&
$\alpha (x) + \alpha (x ^{-1})$&$0$ \\
      \end{tabular}
    \end{center}
 
\subsubsection*{\textnormal{\textbf{The pair of representations $X^\pm$:}}}
The two $(q-1)/2$-dimensional representations $X^+$ and $X^-$ have the same trace character values for exactly all group elements outside of $[c_2 (\epsilon,\gamma)]$, so we have 
$\delta (X^+, X^-) = {2}/{q}.$

\subsubsection*{\textnormal{\textbf{The pair of representations $W^\pm$:}}}
The two $(q + 1)/2$-dimensional representations $W^+$ and $W^-$ have the same trace character values exactly for all group elements outside of the $[c_2 (\epsilon,\gamma)]$ conjugacy classes. So again we have
$\delta (W^+, W^-) = {2}/{q}.$

\subsubsection*{\textnormal{\textbf{$(q-1)$-dimensional representations:}}}\label{pi-a}
There are $(q-1)/2$ such representations, denoted $X_\varphi$, where $\varphi \in \widehat{\E^1}$, for $\varphi ^2 \neq 1$. Note that $|\E^1| = q + 1$.
 
In order to determine $\delta (X_\varphi, X_{\varphi'})$, we need to find the number of $z \in \E^1$ for which $\varphi (z) + \varphi (z^{-1}) = \varphi' (z) + \varphi' (z^{-1})$, and whether $\varphi (-1) = \varphi' (-1)$.

We begin with the first equation. Note that ${\rm Im} (\varphi), {\rm Im} (\varphi') \subset \mu_{q+1}$, where $\mu_n$ denotes the $n$th roots of unity.  Then $\varphi (z) + \varphi (z^{-1} )$ is of the form $\zeta^a + \zeta^{-a}$, where $\zeta$ is the primitive $(q+1)$th root of unity $e^ {2\pi i / (q+1)}$ and $a$ is a non-negative integer less than $q+1$. Now $\zeta^a + \zeta^{-a} = \zeta^b + \zeta^{-b}$ for some $0 \leq a,b < q+1$ implies that $a = b$ or $(q+1) -b$.
So $\varphi (z) + \varphi (z^{-1}) = \varphi' (z) + \varphi' (z^{-1})$ iff $\varphi (z) = \varphi'(z)$ or $\varphi(z) = \varphi' (z^{-1})$.

If $\varphi (z) = \varphi' (z)$, then this is equivalent to $(\varphi') ^{-1} \varphi(z) = 1$, and the number of $z$ for which this holds is $| {\rm ker}\,  (\varphi') ^{-1} \varphi|$. The number of $z$ for which $\varphi (z) = \varphi' (z ^{-1})$ is $| {\rm ker}\, \varphi' \varphi |$. Thus the number of $z \in
\E^1$ for which $\varphi (z) + \varphi (z^{-1}) = \varphi' (z) + \varphi'(z ^{-1})$ is 
\begin{align*}
| {\rm ker}\, (\varphi') ^{-1} \varphi | + |{\rm ker}\, \varphi' \varphi | - |{\rm ker}\, \varphi' \varphi \cap {\rm ker}\, (\varphi')^{-1} \varphi |.
\end{align*}

Now $\E^1$ is a cyclic group, so we can fix a generator $g$. The elements of $\widehat{\E^1}$ can then be denoted as $\{\varphi_0, \varphi_1, \varphi _2, \dots, \varphi_{q}\}$, where $\varphi_m$ is defined via $\varphi_m (g) = \zeta^m$. Note that 
$|{\rm ker}\, \varphi_m |= (m, q+1).$
Define 
\begin{align*}
M_{k}(m,m') := \frac{(m + m', k) + (m-m', k) - (m+m', m-m',k) - 1 - t_{m,m'}}{2},
\end{align*}
where $t_{m,m'}=1$ if both $k$ and $m + m'$ are even, and $0$ otherwise.
 
Then:
\begin{lemma}
For distinct integers $0 \leq m, m' < q+1$, we have 
\[ |\{[c_4(z)] : \varphi_m (z) + \varphi_m (z ^{-1}) = \varphi_{m'} (z) + \varphi_{m'} (z ^{-1}) \}| = M_{q + 1} (m, m'). \]
\end{lemma}
 
If $m$ and $m'$ have the same parity, then $\varphi_m (-1) = \varphi_{m'} (-1)$ so
\begin{align}  \label{eq:Xphi-even}
\delta (X_{\varphi_m},X_{\varphi_{m'}}) 
=\frac{1}{q+1} \left(\frac{q-1}{2} -  M_{q+1}(m,m')\right).
\end{align}

If $m$ and $m'$ have different parity, then
\begin{align} \label{eq:Xphi-odd}
\delta (X_{\varphi_m},X_{\varphi_{m'}})
= \frac{1}{q ^2 -1}\left(\frac{q ^2 +1}{2} -  M_{q+1}(m,m')(q - 1)\right).
\end{align}

To determine the minimum possible value of $\delta$ above, we consider the maximum possible size of $M_k (m,m')$.

\begin{lemma} \label{max-lem}
Suppose $k=2^j \ge 8$.  Then 
\[ \max M_k(m,m') = 2^{j-2} -1 = \frac k4-1, \]
where $m, m'$ run over distinct classes in $\Z/k\Z \setminus \{ 0, \frac k2 \}$ with $m \not 
\equiv \pm m'$.

Suppose $k \in 2 \mathbb N$ is not a power of 2 and let $p$ be the smallest odd
prime dividing $k$.   Then
\[ \max M_k(m,m') = \begin{cases}
\frac k4 \left( 1 + \frac 1 p \right) - 1 & k \equiv 0 \pmod 4 \\
\frac{k-2}4 & k \equiv 2 \pmod 4, \\
\end{cases}  \]
where $m, m'$ range as before.

In all cases above, the maximum occurs with $m, m'$ of the same parity if and only if $4 | k$.
\end{lemma}

\begin{proof} Let $d = (m+m', k)$ and $d' = (m-m', k)$, so our restrictions on $m, m'$
imply that $d, d'$ are proper divisors of $k$ of the same parity.  
Note that any pair of such $d, d'$ arise from some $m, m'$
if $d \ne d'$, and the case $d = d' = \frac k2$ does not occur.
 Then $M_k(m,m') = \frac 12 ( d + d' - (d, d') -1 - t_{m,m'})$, and
$m, m'$ have the same parity if and only if $d, d'$ are both even.

The case $k=2^j$ has a maximum with $d = \frac k2$ and $d'= \frac k4$.

Suppose $k=2pk'$ as in the second case.  Then
note $d + d' - (d, d')$ is maximised when $d = \frac k2$ and $d' = \frac kp$,
which is an admissible pair if $k'$ is even.  Otherwise, we get a maximum when $d = \frac k2$ and $d' = \frac k{2p}$.
\end{proof}

In all cases we have 
\begin{equation} \label{eq:max-bound}
\max M_k(m,m') \le \frac k3 - 1,
\end{equation} and equality is obtained if and
only if $12 | k$ for suitable $m, m'$ of the same parity.
This leads to an exact formula for $\delta_{q-1}(\SL(2,q))$ with $q > 3$ odd by
combining with \eqref{eq:Xphi-even} and \eqref{eq:Xphi-odd}.  We do not
write down the final expression, but just note the consequence that
$\delta_{q-1}(\SL(2,q)) \ge \frac 16$ with
equality if and only $12 | (q+1)$.

\subsubsection*{\textnormal{\textbf{$(q + 1)$-dimensional representations:}}}\label{rho-a}
Consider $W_\alpha, W_{\alpha'}$, where $\alpha, \alpha' \in \widehat{\F_q^\times} - \{\pm 1\}$ and $\alpha \neq \alpha'$.
Since $ |\F_q^\times| = q - 1$, we know that ${\rm Im}(\alpha) < \mu_{q - 1}$. So, given a generator $g$ of the cyclic group $\F_q^\times$, we define the elements of $\widehat{\F_q^\times}$ as: $\alpha_m (g) = \zeta^m$, where $\zeta := e ^{2\pi i / (q - 1)}$, and $0 \leq m \leq q-2$. \\

Using similar arguments to the $(q-1)$-dimensional case above, we have:
\begin{lemma}
For distinct integers $0 \leq m, m' < q-1$, we have 
\begin{align*}
&\left| \{[c_3(x)] : \alpha_m (x) + \alpha_m (x ^{-1}) = \alpha_{m'} (x) + \alpha_{m'} (x ^{-1}) \} \right|
= M_{q-1}(m,m').
\end{align*}
\end{lemma}

Given that the value of $\alpha_m (-1)$ is $+1$ if $m$ is even and $-1$ if $m$ is odd, we obtain
that if $m$ and $m'$ have the same parity, then 
\begin{align*}
\delta (W_{\alpha_m},W_{\alpha_{m'}}) 
=\frac{1}{q-1} \left(\frac{q - 3}{2} -  {M_{q-1}(m,m')}{}\right).
\end{align*}
Whereas if $m$ and $m'$ have different parity, then 
\begin{align*}
\delta (W_{\alpha_m},W_{\alpha_{m'}})  = \frac{1}{q ^2 -1}\left(\frac{q ^2 -3}{2} -  {M_{q-1}(m,m')}{}(q + 1)\right).
\end{align*}

Combining these with Lemma \ref{max-lem} for $q > 5$
gives a formula for $\delta_{q+1}(\SL(2,q))$.  In particular,
\eqref{eq:max-bound} gives $\delta_{q+1}(\SL(2,q)) \ge \frac 16$,
with equality if and only if $12 | (q-1)$.

\subsection{$\SL(2,q)$, for even $q$}

We keep the notation from the previous section. The order of $\SL(2,q)$ is again $q(q + 1)(q-1)$.
The conjugacy classes for $\SL(2,q)$, $q$ even, are as follows: 

\begin{enumerate}[(A)]
\item $I$. 
\item $[N]=     \left[\left( \begin{array}{cc}
     1&1 \\
     0&1
     \end{array} \right)\right]$. This conjugacy class is of size $q ^2 -1$.
\item $[c_3 (x)]$, where
$c_3 (x) =      \left( \begin{array}{cc}
x& \\
&x ^{-1}
\end{array} \right)$, with $x \neq 1$. We note that $[c_3 (x)] = [c_3 (x ^{-1})]$, so there are $(q-2)/2$ such conjugacy classes. Each one is of size $q (q + 1)$.
\item $[c_4 (z)]$, where $ c_4 (z) =      \left( \begin{array}{cc}
x&\Delta y \\
y&x
\end{array} \right)$ for $z = x + \delta y \in \E^1$ with $z \neq 1$. Since $c_4 (z) = c_4 (\bar{z})$, there are $q/2$ such conjugacy classes, each of size $q (q-1)$.\\
\end{enumerate}

The representations for $q$ even are constructed similarly to the case of $q$ odd, with a couple of differences: Since, for $q$ even, the subgroup  $S$ has odd order, it does not have characters of order two, and so the irreducible representations $X ^\pm $ do not arise. Similarly, the character $\alpha$ cannot be of order two, and so the irreducible representations $W ^\pm$ do not occur.
The character table is:

    \begin{center}
      \begin{tabular}{r|c|c|c|c|c} 
 &   &$[I]$&$[N]$&$[c_3(x)]$& $[c_4(z)]$ \\ \hline
 & Size: &1 & $q ^2 -1$ & $q(q + 1)$ &  $q (q-1)$\\ \hline 
Rep & \# &&&& \\ \hline
$U$ & 1 & $1$ & 1 & 1 & $1$\\
$X_\varphi$ & $q/2$ & $q-1$& $-1$&0& $-\varphi (z) - \varphi (z ^{-1})$ \\
$V$ & 1 & $q$ & 0 & 1 & $-1$\\
$W_\alpha$ & $(q-2)/2$&$q + 1$&1&$\alpha (x) + \alpha (x ^{-1})$&$0$ \\
      \end{tabular}
    \end{center}

\subsubsection*{\textnormal{\textbf{Representations of dimension $q-1$:}}}

The analysis here is similar to that in Section~\ref{pi-a}, which gives us:
\[ 
\delta (X_{\varphi_m}, X_{\varphi_{m'}} ) = \frac{1}{q + 1}\left(\frac{q}{2} - M_{q + 1}(m,m')\right).
\]

Analogous to Lemma \ref{max-lem}, we have when $k\ge 3$ is odd,
\begin{equation} \label{eq:max-even}
\max M_{k}(m,m') = 
\begin{cases}
\frac 12 \left( \frac kp - 1 \right) & k = p^j \\
\frac 12 \left( \frac k{p_1 p_2}(p_1+p_2-1) - 1 \right) & k = p_1 p_2 k' 
\end{cases}
\end{equation}
where $m, m'$ run over all nonzero classes of $\Z/k\Z$ such that $m \not \equiv \pm m'$
and in the latter case are the two smallest distinct primes dividing $k$.  
The above two equations give an exact expression for $\delta_{q-1}(\SL(2,q))$, $q \ge 4$.
For $k$ odd, note
\begin{equation} \label{eq:max-bound-even}
\max M_{k}(m,m') \le \frac{7k-15}{30},
\end{equation}
with equality if and only if $15 | k$.
Thus $\delta_{q-1}(\SL(2,q)) \ge \frac 4{15}$ with equality if and only if $15 | (q+1)$.

\subsubsection*{\textnormal{\textbf{Representations of dimension $q + 1$:}}}

A similar analysis to that in Section~\ref{rho-a} gives 
\[ 
\delta (W_{\alpha_m}, W_{\alpha_{m'}} ) = \frac{1}{q-1}\left(\frac{q-2}{2} - M_{q - 1}(m,m') \right).
\]
Combining this with \eqref{eq:max-even} gives an exact formula for $\delta_{q+1}(\SL(2,q))$
for $q \ge 8$,
and from \eqref{eq:max-bound-even}, we again get $\delta_{q+1}(\SL(2,q)) \ge \frac{4}{15}$
with equality if and only if $15 | (q-1)$.

\subsection{$\PSL(2,q)$, for odd $q$}
The order of $\PSL(2,q)$ is $\frac{1}{2}q (q ^2 -1)$ if $q$ is odd.
The conjugacy classes are as follows:
\begin{enumerate}[(A)]
\item $I$.
\item $[c_2 (\gamma)]$, where $c_2(\gamma) = c_2 (1,\gamma)=      \left( \begin{array}{cc}
1&\gamma \\
&1
\end{array} \right)$ for $\gamma \in \{1, \Delta\}$.
\item $[c_3 (x)]$, $(x \neq \pm 1)$, where $c_3 (x)$ is as in the previous two sections. Since $c_3 (x) = c_3 (-x) = c_3 (1/x) = c_3 (-1/x)$, the number of such conjugacy classes when $q \equiv 3 \pmod 4$ is $(q-3)/4$. In this case, all of the $c_3 (x)$ conjugacy classes have size $q (q + 1)$.

If $q \equiv 1 \pmod 4$, then $-1$ is a square in $\F_q$ and there is a conjugacy class denoted by $c_3 (\sqrt{-1})$ which has size $q (q + 1)/2$; the remaining $c_3 (x)$ conjugacy classes (there are  $(q-5)/4$ such classes) have size $q (q + 1)$.

\item $[c_4 (z)]$, for $z \in \E^1, z \neq \pm 1$, where $c_4 (z)$ is defined as in the previous two sections. Since $c_4 (z) = c_4 (\bar{z}) = c_4 (-z) = c_4 (- \bar{z})$, when $q \equiv 1 \pmod 4$, the number of such conjugacy classes is $(q-1)/4$, and they are all of size $q (q-1)$. When $q \equiv 3 \pmod 4$, we can choose $\Delta$ to be $-1$ (since it is not a square), and so we see that $\delta \in \E^1$. The conjugacy class associated to $c_4 (\delta)$ has size $q (q-1)/2$, whereas the rest of the $c_4 (z)$ conjugacy classes (of which there are $(q-3)/4$ such classes) have size $q (q-1)$.\\
\end{enumerate}

The representations of $\PSL(2,q)$ are the representations of $\SL(2,q)$ which are trivial on $-I$; this depends on the congruence class of $q$ modulo 4.

\subsubsection{$q \equiv 1 \pmod 4$}\ \\
For the character table below, the notation is the same as in previous subsections.
    \begin{center}
      \begin{tabular}{r|c|c|c|c|c|c} 
&   &$[I]$&$[c_2(\gamma)]$&$[c_3(\sqrt{-1})]$&$[c_3(x)]$& $[c_4(z)]$ \\ \hline
&Size:&1& $\frac{q ^2 -1}{2}$ & $\frac{q (q + 1)}{2}$&$q (q + 1)$& $q (q-1)$\\ \hline 
Rep & $\#$ &1&2&1&$\frac{q-5}{4}$&$\frac{q-1}{4}$ \\ \hline
$U$ & 1 & 1 &1&1&1&1 \\
$W^\pm$ & 2 & $\frac{q + 1}{2}$ & $s^\pm (1,\gamma)$&$\tau (\sqrt{-1})$ &$\tau (x)$& 0 \\
$X_\varphi$ & $\frac{q-1}{4}$ & $q-1$& $-1$&0&0& $-\varphi (z) - \varphi (z ^{-1})$ \\
$V$ & 1 & $q$ &  0 & 1&1 & $-1$\\
$W_\alpha$ & $\frac{q-5}{4}$&$q + 1$&1&$2\alpha (\sqrt{-1})$&$\alpha (x) + \alpha (x ^{-1})$&$0$ \\
      \end{tabular}
    \end{center}
\ \\

\subsubsection*{\textnormal{\textbf{Representations $W ^\pm$:}}}

The trace characters of these $(q + 1)/2$-dimensional representations agree everywhere but for the conjugacy classes $[c_2 (\gamma)]$. This gives us 
$\delta (W^+,W^-) = {2}/{q}.$

\subsubsection*{\textnormal{\textbf{Representations of dimension $q -1$:}}}

Assume $q \ge 9$.
Any two representations $X_\varphi, X_{\varphi'}$ have trace characters that may differ only for the conjugacy classes $[c_4 (z)]$.  We may view $\varphi$ as a map into $\mu_{\frac{q+1}2}$
and parameterize the $\varphi$ by $\varphi_m$ for nonzero $m \in \Z/\frac{q+1}2 \Z$ similar to 
before. Analogously, we obtain 
\[
\delta (X_{\varphi_m}, X_{\varphi_{m'}}) = \frac{1}{q + 1} \left(\frac{q-1}{2} - 2M_{\frac{q + 1}2}(m,m')\right).
\]
From \eqref{eq:max-bound-even}, this gives $\delta_{q-1}(\PSL(2,q)) \ge \frac 4{15}$,
with equality if and only if $30 | (q+1)$.

\subsubsection*{\textnormal{\textbf{Representations of dimension $q + 1$:}}}

Assume $q \ge 13$.
The analysis follows in a similar manner to that in previous sections.
View $\alpha : \F_q^\times / \{ \pm 1 \} \to \mu_{\frac{q-1}2}$, and we can parametrize such
$\alpha$ by $m \in \Z/\frac{q-1}2 \Z$ as before.
 One difference is that we must consider the case when $x = \sqrt{-1}$. Note that this is the only conjugacy class of the form $[c_3 (x)]$ that has size $q (q + 1)/2$. We find that 
$\alpha_m (\sqrt{-1}) = \alpha_{m'}(\sqrt{-1})$
if and only if $m, m'$ have the same parity.
Overall we get 
\[
\delta (W_{\alpha_m}, W_{\alpha_{m'}}) = \frac{1}{q -1} \left(\frac{q-5}{2} - 2M_{\frac{q-1}2}(m,m')
+ 1 - t_{m,m'} \right).
\]
From \eqref{eq:max-bound} we get $\delta_{q+1}(\PSL(2,q)) \ge \frac 16$ with equality
if and only if $24 | (q-1)$.

\subsubsection{$q \equiv 3 \pmod 4$}

    \begin{center}
      \begin{tabular}{r|c|c|c|c|c|c} 
&   &$[I]$&$[c_2(\gamma)]$&$[c_3(x)]$& $[c_4(z)]$&  $[c_4(\delta)]$ \\ \hline
&Size:&1& $\frac{q ^2 -1}{2}$ & $q (q + 1)$ & $q (q-1)$ & $\frac{q (q - 1)}{2}$\\ \hline 
Rep & Number &1&2&$\frac{q-3}{4}$&$\frac{q-3}{4}$&1 \\ \hline
$U$ & 1 & 1 &1&1&1&1 \\
$X^\pm$ & 2 & $\frac{q-1}{2}$ & $u^\pm (1,\gamma)$& 0&$-\psi (z)$&  $-\psi (\delta)$\\
$X_\varphi$ & $\frac{q-3}{4}$ & $q-1$& $-1$&0& $-\varphi (z) - \varphi (z ^{-1})$& $-2 \varphi (\delta)$ \\
$V$ & 1 & $q$ &  0 & 1 & $-1$&1\\
$W_\alpha$ & $\frac{q-3}{4}$&$q + 1$&1&$\alpha (x) + \alpha (x ^{-1})$&$0$&0
 \end{tabular}
    \end{center}
\ \\
where $u^\pm(1,\gamma)$ and $\psi$ are defined as before.\\

\subsubsection*{\textnormal{\textbf{Representations $X^\pm$:}}}

For $W^\pm$, the characters of the representations $X^\pm$ agree everywhere but for the conjugacy classes $[c_2 (\gamma)]$, so:
$\delta (X^+, X^-) = {2}/{q}.$

\subsubsection*{\textnormal{\textbf{Representations of dimension $q-1$:}}}
Assume $q \ge 11$.
Any two representations $X_\varphi, X_{\varphi'}$ have trace characters that may differ only for the conjugacy classes $[c_4 (z)]$. 
In the case of the conjugacy class $[c_4 (\delta)]$, we note that $\delta$ has order 2 in 
$\E^1/\{ \pm 1 \}$.  Parametrize the nontrivial maps 
$\varphi: \E^1/ \{ \pm 1 \} \to \mu_{\frac{q+1}2}$ by $1 \le m \le \frac{q-3}4$ as before.
Then $\varphi_m (\delta) = \varphi_{m'}(\delta)$ if and only if $m, m'$
have the same parity.
We obtain 
\begin{align*}
\delta (X_{\varphi_m}, X_{\varphi_{m'}}) = \frac{1}{q + 1} \left(\frac{q-3}{2}  - 2M_{\frac{q+1}2}(m,m') + 1 - t_{m,m'} \right).
\end{align*}
By \eqref{eq:max-bound}, we get $\delta_{q-1}(\PSL(2,q)) \ge \frac 16$ with equality if and
only if $24 | (q+1)$.

\subsubsection*{\textnormal{\textbf{Representations of dimension $q + 1$:}}}

Assume $q \ge 11$.
We obtain
\[
\delta (W_{\alpha_m}, W_{\alpha_{m'}}) = \frac{1}{q -1} \left(\frac{q-3}{2} -  2M_{\frac{q-1}2}(m,m')\right).
\]
By \eqref{eq:max-bound-even}, we get $\delta_{q+1}(\PSL(2,q)) \ge \frac 4{15}$,
with equality if and only if $30 | (q-1)$.

\end{document}